\documentclass[12pt,reqno]{amsart} 
\usepackage{amssymb}

\newtheorem{theorem}{Theorem}[section]
\newtheorem{lemma}[theorem]{Lemma}
\newtheorem{proposition}[theorem]{Proposition}
\newtheorem{corollary}[theorem]{Corollary}

\theoremstyle{definition}
\newtheorem{definition}[theorem]{Definition}
\newtheorem{remark}[theorem]{Remark}

\numberwithin{equation}{section}

\begin{document}

\baselineskip=15pt

\title[Moduli spaces of vector bundles]{Moduli spaces of vector bundles 
on a singular rational ruled surface}

\author[U. N. Bhosle]{Usha N. Bhosle}

\address{Department of Mathematics, Indian Institute of Science, Bangalore, India}

\email{usnabh07@gmail.com}

\author[I. Biswas]{Indranil Biswas}

\address{School of Mathematics, Tata Institute of Fundamental
Research, Homi Bhabha Road, Mumbai 400005, India}

\email{indranil@math.tifr.res.in}



\date{}

\thanks{This work was finalized during the first author's tenure as Raja Ramanna 
Fellow at the Indian Institute of Science, Bangalore. The second author is 
supported by J. C. Bose Fellowship.}

\begin{abstract}
We study moduli spaces $M_X(r,c_1,c_2)$ parametrizing slope semistable vector bundles of 
rank $r$ and fixed Chern classes $c_1, c_2$ on a ruled surface whose base is a 
rational nodal curve. We show that under certain conditions, these moduli 
spaces are irreducible, smooth and rational (when non-empty). We also prove that they 
are non-empty in some cases.

We show that for a rational ruled surface defined over real numbers, the moduli
space $M_X(r,c_1,c_2)$ is rational as a variety defined over $\mathbb R$.
\end{abstract}

\maketitle

\section{Introduction}

Vector bundles on smooth complex ruled surfaces have been studied by many authors from 
different points of view, the case of rank two being studied most extensively. Let $X$ be a complex
projective surface equipped with a polarization $H$, and let $M_{X,H}(r, 
c_1,c_2)$ denote the moduli space of $H$-semistable (slope semistable) vector bundles on $X$ of rank $r$ with 
fixed Chern classes $c_1 \,\in\, {\rm Pic}(X)$ and $c_2 \,\in\, \mathbb{Z}$. 
When $X$ is a smooth ruled surface or a blow up of it, Walter, \cite{W}, found a precise sufficient condition
on $H$ for $M_{X,H}(r, c_1,c_2)$ to be irreducible whenever it is non-empty.
He also proved that this moduli space is normal and its subvariety $M^s_{X,H}(r,c_1,c_2)$ that 
parametrizes stable vector bundles is smooth. Furthermore, he gave examples of $X$ and $H$ (not satisfying
his condition) for which the moduli spaces $M_{X,H}(2, c_1,c_2)$ are reducible for
some small $c_1, c_2$.

Another interesting property investigated by many is the rationality of
the scheme $M_{X,H}(r, c_1,c_2)$. The 
question is the following: If $X$ is rational, is $M_{X,H}(r, c_1,c_2)$ also rational?
Although several cases are known where the answer is 
positive, the answer is not known in general. Costa and Mir\'o-Roig 
explicitly constructed generic $H$-stable vector bundles on a smooth Hirzebruch surface
$X$ for many values of $r\, ,c_1\, , c_2$, \cite{CM}, and showed non-emptiness for these
moduli spaces. They also proved that the moduli space is a rational variety in these
cases \cite[Theorem A]{CM}.

In this paper, we generalize these results to singular rational ruled surfaces. Let $X$ be a complex 
rational ruled surface whose base is an integral rational projective curve $Y$ of arithmetic genus $g$
with $g$ nodes (ordinary double points) as its only singularities. We study the geometric 
properties of $X$. In particular, we prove that $X$ is a Gorenstein variety, compute its invariants 
and determine the dualizing sheaf explicitly. Following \cite{W}, we establish a sufficient condition 
on $H$ for the moduli space $M_{X,H}(r,c_1,c_2)$ to be irreducible. Furthermore, we prove the existence
of polarizations $H$ satisfying this condition (see Theorem \ref{thm1}).

We also investigate the rationality question for $M_{X,H}(r,c_1,c_2)$. Let $\pi \colon \overline{Y} 
\longrightarrow Y$ be the normalization map for the base curve. If $X\,:=\, {\mathbb P}({\mathcal E})$, then $Z\,:=\, 
{\mathbb P}(\pi^*{\mathcal E})$ is a smooth Hirzebruch surface. Let $H_Z$ denote the polarization on $Z$ 
which is the pullback of the polarization $H$ on $X$. We show that whenever $M^s_{Z,H_Z}(r,c_1,c_2)$ is rational, the 
variety $M^s_{X,H}(r,c_1,c_2)$ is also rational (see Theorem \ref{theorem2}). In view of the results of 
\cite{CM}, this yields rationality of $M^s_{X,H}(r,c_1,c_2)$ in several cases.

Finally we study singular real rational ruled surfaces $X_{\mathbb R}$ whose
base $Y_{\mathbb R}$ a rational curve defined over $\mathbb R$. Let $\pi'\,:\, C\, \longrightarrow\, Y_{\mathbb R}$ be the normalization map. Let
${\mathcal E}_{\mathbb R}$ be a real vector bundle of rank $2$ over $Y_{\mathbb R}$ such that
$X_{\mathbb R}\, =\, \mathbb{P} ({\mathcal E}_{\mathbb R})$. Let $Z_{\mathbb R}\, := \mathbb{P} (\pi'^* {\mathcal E}_{\mathbb R})$ be 
the real ruled surface with base $C$. For a real ruled surface $Z_{\mathbb R}$ with base an 
anisotropic conic $C$, the rationality question of $M^s_{Z_{\mathbb R},H_Z}(r,c_1,c_2)$ was studied in \cite{BS}.

We prove that $M^s_{X_{\mathbb R},H}(r,c_1,c_2)$ is a real rational variety if $M^s_{Z_{\mathbb 
R},H_Z}(r,c_1,c_2)$ is a real rational variety (see Theorem \ref{theorem3}). Coupled with the results of 
\cite{BS}, this gives necessary and sufficient conditions for $M^s_{X_{\mathbb R},H}(r,c_1,c_2)$ to be 
a real rational variety (Theorem \ref{theorem4}).

\section{Singular ruled surfaces}

In this section, we define a singular rational ruled surface that we are 
interested in and study its properties.

\subsection{Notation}\label{notation2.1}

Let $Y$ be an integral rational projective curve of arithmetic genus $g$ over $\mathbb C$
with only nodes (ordinary double points) as singularities. Therefore, $Y$ has exactly
$g$ singular points. Let
$y_1\, , \cdots\, , y_g$ be the singular points of $Y$. Let
$$\pi\,:\, {\overline Y} \,\longrightarrow \, Y$$ 
be the normalization map. Then ${\overline Y}\,=\, \mathbb{P}^1_{\mathbb C}$ because
$Y$ is rational. For $1\,\leq\, j\, \leq\, g$, let $\{x_j\, ,z_j\}\,\subset\,
\overline Y$ be the pair of points over $y_j\,\in\, Y$. 
 
Take an algebraic vector bundle $\mathcal E$ over $Y$ of rank two and degree $-e$. Let 
$$X\,:=\, {\mathbb P}({\mathcal E})\,:\, X \stackrel{p_1}{\longrightarrow}\, Y$$
be the corresponding ${\mathbb P}^1_{\mathbb C}$--bundle over $Y$. Then 
$$Z\,:=\, {\mathbb P}({\mathcal E})\times_Y {\overline Y}\,=\,
{\mathbb P}(\pi^*{\mathcal E})$$
is a smooth Hirzebruch surface. Let $$p_0\,:\, Z \,\longrightarrow\,{\overline Y}$$ be the
projection to the second factor of the fiber product. The relatively ample tautological line bundles on $X$ and $Z$
will be denoted by ${\mathcal O}_{p_1}(1)$ and ${\mathcal O}_{p_0}(1)$ respectively. 

We fix an ample line bundle $H$ on $X$. Let $$p\,:\, Z\,=\, X\times_Y\overline{Y}\,
\longrightarrow\, X$$ be the projection to the first factor of the
fiber product. Define the line bundle
$$H_Z\,:=\, p^*H\,\longrightarrow\, Z\, .$$ Since $p$ is a finite map, and
$H$ is ample, it follows that $H_Z$ is ample.

By tensoring ${\mathcal E}$ with a line bundle, we may assume that $Z\,=\,
{\mathbb P} ({\mathcal O}_{\overline Y} \oplus {\mathcal L})$, because on one hand
$\pi^*{\mathcal E} \,\cong \,{\mathcal L}_1 \oplus {\mathcal L}_2 \,\cong\,
{\mathcal L}_1 \otimes ({\mathcal O}_{\overline Y} \oplus
({\mathcal L}_2 \otimes {\mathcal L}_1^{-1}))$, on the other
hand, ${\mathcal L}_1\,=\, \pi^*N$ for some line bundle $N$ on $Y$,
so $\pi^*({\mathcal E} \otimes N^{-1}) \,\cong\, {\mathcal O}_{\overline Y} \oplus
({\mathcal L}_2 \otimes {\mathcal L}_1^{-1})$. The inclusion of ${\mathcal
O}_{\overline Y}$ in ${\mathcal O}_{\overline Y} \oplus ({\mathcal L}_2 \otimes
{\mathcal L}_1^{-1})$ defines an irreducible effective divisor on $Z$; we denote
this divisor by $C_0$.
 
We have a commutative diagram
 $$
 \begin{array}{ccc}
 Z & \stackrel{p}{\longrightarrow} & X \\
 p_0\Big\downarrow ~\,\, ~  &   &  p_1\Big\downarrow~\,\, ~ \\
 {\overline Y}  &  \stackrel{\pi}{\longrightarrow}  &  Y 
 \end{array}
 $$
For $1\, \leq \, j\, \leq\, g$, define $$P_j\,:=\, p_1^{-1}(y_j)\, , ~\ P_{x_j}\,:=\,
p_0^{-1}(x_j)\, , ~\ P_{z_j}\,:= \,p_0^{-1}(z_j)\, ,$$
so $p^{-1}(P_j) \,= \,P_{x_j} \coprod P_{z_j}$. The restrictions $p\vert_{P_{x_j}}$ and
$p\vert_{P_{z_j}}$ identify $P_j$ with $P_{x_j}$ and $P_{z_j}$
respectively. Therefore, we obtain a canonical isomorphism 
\begin{equation}\label{tau}
\tau_j\,:\, P_{x_j}\, \stackrel{\sim}{\longrightarrow}\, P_{z_j}\, .
\end{equation}
We note that $\tau_j$ is induced by the canonical identification of $(\pi^* {\mathcal E})_{x_j}$ with $(\pi^* {\mathcal E})_{z_j}$.

The Hirzebruch surface $Z$ has been studied extensively (see \cite[Chapter V]{Ha} for generalities on Hirzebruch surfaces). We start with some geometric properties of $X$.

\begin{lemma}\label{l2.1} \mbox{}
\begin{enumerate}
\item The variety $X$ is semi-normal; the disjoint union $\bigcup_{j=1}^g P_j$
is the non-normal locus.
\item The variety $X$ is Gorenstein.
\item The dualizing sheaf $\omega_X$ of $X$ is a locally free sheaf of rank one.
\end{enumerate}
\end{lemma}

\begin{proof}
(1):\, As the singularities of $Y$ are ordinary nodes, $Y$ is a semi-normal variety. 
Locally, $X$ is a product of a semi-normal variety with a normal (in fact non-singular)
variety and hence $X$ is a semi-normal variety 
\cite[Proof of Corollary 5.9]{GT}.

Since $p^{-1}(Y - \bigcup_{j=1}^g y_j)$ is non-singular, the last assertion in (1) follows.

(2):\, The fibers of $p_1$ are non-singular and hence Gorenstein. The
morphism $p_1$ is flat, the base $Y$ is Gorenstein, and the
fibers of $p_1$ are also Gorenstein. Therefore, it follows that $X$ 
is Gorenstein \cite[Proposition 9.6]{Ha1}.

(3):\, Since $X$ is Gorenstein (by part (2)), the dualizing sheaf $\omega_X$ is a
locally free sheaf of rank $1$ \cite[p. 295--296, Theorem 9.1]{Ha1}.
\end{proof}

The following lemma, which sums up facts about $X$, is an easy but a useful one.

\begin{lemma} \label{l2.2} \mbox{}
\begin{enumerate}
\item The Picard group ${\rm Pic}(X) \,=\, p^*_1 {\rm Pic}(Y) \oplus
{\mathbb Z} {\mathcal O}_{p_1}(1)$.

\item Invariants of $X$: \\
Arithmetic genus $ p_a(X)\,:= \,\chi({\mathcal O}_X) - 1 \,= \,-g$,\\
geometric genus $p_g(X) \,:=\, H^2(X,\, {\mathcal O}_X) \,=\,0$, and \\
irregularity $q(X)\,:=\, H^1(X,\, {\mathcal O}_X) \,=\, g$.

\item $h^0(X,\, \omega_X)\,=\,0$, $h^1(X, \,\omega_X)\,=\,g$ and $h^2(X,\,
\omega_X)\,=\,1$.
\end{enumerate}
\end{lemma}

\begin{proof}
Since $p_1\,:\, X \,\longrightarrow\, Y$ is a ${\mathbb P}^1$-bundle, the
first statement follows.

Since $(p_1)_*{\mathcal O}_X \,=\, {\mathcal O}_Y$, we have
$h^i(X,\, {\mathcal O}_X) \,= h^i(Y, \,{\mathcal O}_Y)~\, \forall ~ i$,
hence $\chi({\mathcal O}_X) \,= \,\chi({\mathcal O}_Y)$. Therefore (2) follows.

Statement (3) follows from (2) and Serre duality.
\end{proof}

\begin{remark}\label{r2.3}
For any $y\,\in\, Y$, the fiber $p_1^{-1}(y)$ is isomorphic
to ${\mathbb P}^1$. However, if $y$ is a 
non-singular point, then the fiber is a Cartier divisor, otherwise it is not a Cartier 
divisor. For $y$ non-singular, the Cartier divisor 
$$F_y\,:=\, p_1^{-1}(y)$$
corresponds to the line bundle $p_1^*{\mathcal O}_Y(y)$.
For a node $y\,=\,y_j$, the fiber $F_{y_j}$ is locally defined by two equations. The ideal 
sheaf $I(F_{y_j})$ of $F_{y_j}$ in ${\mathcal O}_X$ is isomorphic to $p_1^* I(y_j)$, 
where $I(y_j)$ denotes the ideal sheaf of $y_j$ in ${\mathcal O}_Y$.
\end{remark}
 
\begin{proposition} \label{p2.4}
The dualizing sheaf $\omega_X$ is isomorphic to
$p_1^*(\omega_Y \otimes \det {\mathcal E}) \otimes {\mathcal O}_{p_1}(-2)$.
\end{proposition}

\begin{proof}
Since ${\overline Y}$ and $Y$ are Gorenstein curves, and $\pi$ is a finite map, we
conclude that 
$\omega_{{\overline Y}} \,\cong\, \pi^* \omega_Y\otimes {\mathcal C}_{{\overline Y}/Y}$.
As the conductor sheaf ${\mathcal C}_{{\overline Y}/Y}$ is isomorphic to
${\mathcal O}_{\overline Y}(-\sum_j (x_j+z_j))$, we have 
$$\omega_{{\overline Y}} \otimes{\mathcal O}_{\overline Y}(\sum_j (x_j+z_j) )\,\cong
\, \pi^* \omega_Y\, .$$

Similarly, since $X$ and $Z$ are both Gorenstein varieties (see Lemma
\ref{l2.1}), and $p$ is a finite map between them, one has
$$\omega_Z \,\cong \,p^* \omega_X\otimes {\mathcal C}_{Z/X}\, ,$$
where ${\mathcal C}_{Z/X}$ is the conductor sheaf. We have
${\mathcal C}_{Z/X} \,\cong \,p_0^*{\mathcal C}_{{\overline Y}/Y}$. Therefore,
it follows that ${\mathcal C}_{Z/X} \,\cong \,{\mathcal O}_Z (- \sum_j (P_{x_j}+P_{z_j}))$. 
Hence
$$p^* \omega_X \cong \omega_Z \otimes {\mathcal O}_Z(\sum_j (P_{x_j}+P_{z_j}))\, .$$
It is known that $\omega_Z\,\cong\, p_0^*(\omega_{\overline Y} \otimes (\det
\pi^*{\mathcal E})) \otimes {\mathcal O}_{p_0}(-2)$ \cite[Chapter V, 
Lemma 2.10]{Ha}. Hence 
$$
\begin{array}{ccl}
p^*\omega_X & \cong & p_0^*(\omega_{\overline Y} \otimes (\det \pi^*{\mathcal E})
\otimes {\mathcal O}_{\overline Y}(\sum_j(x_j+z_j))) \otimes 
{\mathcal O}_{p_0}(-2)\\
{} & \cong & p_0^*(\pi^*\omega_Y \otimes (\det \pi^*{\mathcal E})) \otimes
{\mathcal O}_{p_0}(-2)\\
{} & \cong & p_0^*(\pi^*(\omega_Y \otimes (\det {\mathcal E}))) \otimes
p^*{\mathcal O}_{p_1}(-2)\\
{} & \cong & p^*(p_1^*(\omega_Y \otimes (\det {\mathcal E}))) \otimes  p^*{\mathcal O}_{p_1}(-2)\\
{} & \cong & p^*(p_1^*(\omega_Y \otimes (\det {\mathcal E})) \otimes {\mathcal O}_{p_1}(-2))\, .
\end{array}
$$ 
Thus $$p^* \omega_X \cong p^*(p_1^*(\omega_Y \otimes (\det {\mathcal E}))\otimes
{\mathcal O}_{p_1}(-2))\, .$$
This, combined with the facts that $${\rm Pic}(X)\, \cong\, p^*_1 {\rm Pic}(Y)
\oplus {\mathbb Z} {\mathcal O}_{p_1}(1)\, , \ {\rm Pic}(Z)\,\cong \,p^*_0 {\rm Pic}(
\overline{Y})
\oplus {\mathbb Z} {\mathcal O}_{p_0}(1)$$ (see Lemma \ref{l2.2}(1)) and 
$p^* {\mathcal O}_{p_1}(1) \,\cong\, {\mathcal O}_{p_0}(1)$, implies 
that 
$$\omega_X \,\cong \,p_1^*(\omega_Y \otimes(\det {\mathcal E}) \otimes N) \otimes
{\mathcal O}_{p_1}(-2) \, ,$$
for some line bundle $N$ on $Y$ whose pull-back to ${\overline Y}$ is trivial.

Taking direct image on $Y$, we have 
$$R^i(p_1)_* \omega_X\,\cong\, (\omega_Y \otimes (\det {\mathcal E}) \otimes N)
\otimes R^i(p_1)_* {\mathcal O}_{p_1}(-2)\, .$$

Since $H^i(\mathbb{P}^1, \mathcal{O}(-2))\,=\,0 \, \ \forall i\,\neq\, 1$, and $H^1(\mathbb{P}^1, \mathcal{O}(-2)) =1$, 
one has $$R^i(p_1)_* {\mathcal O}_{p_1}(-2)\,=\,0 \, \ \forall i \,\neq\, 1$$ and
$R^1(p_1)_* {\mathcal O}_{p_1}(-2)$ is a line bundle. Tensoring the Euler sequence 
$$0 \longrightarrow (det \ p_1^* \mathcal{E}) \otimes {\mathcal O}_{p_1}(-1) \longrightarrow
p_1^* \mathcal{E} \longrightarrow {\mathcal O}_{p_1}(1) \longrightarrow 0$$
with $\mathcal{O}_{p_1}(-1)$ and taking direct image by $p_1$, one gets 
$R^1(p_1)_* {\mathcal O}_{p_1}(-2)\,\cong \, \det {\mathcal E}^*$. Hence we have
$$(p_1)_*\omega_X\,=\,0\, ,\, \ R^2(p_1)_*\omega_X \,=\,0$$ and the only non-vanishing
direct image is $R^1(p_1)_*\omega_X \,\cong\, \omega_Y \otimes N$.
Consequently, we have
$$h^1(X, \,\omega_X)\,=\, h^0(Y,\,R^1p_{1*}\omega_X)
\,=\,h^0(Y,\,\omega_Y \otimes N)\, .$$ Now
Lemma \ref{l2.2}(3)
implies that $h^0(Y,\, \omega_Y \otimes N)\,=\,g$. By Serre duality, 
$h^1(Y,\, N^*)\,=\,g$. So by Riemann--Roch, we have $h^0(Y,\, N^*)\,=\,1$.
Since $d(N^*)\,=\,0$, from $h^0(Y,\, N^*)\,=\,1$ it
follows that $N^*$, and hence $N$, is the
trivial line bundle. This completes the proof of the proposition.
\end{proof}

\section{Irreducibility of the moduli space $M(r,c_1,c_2)$.}

Our goal in this section is to prove that the moduli scheme $M_{X, H}(r, c_1, 
c_2)$ of $H$-stable vector bundles on $X$ of rank $r$ and with fixed Chern classes 
$c_1, c_2$ is irreducible if $H$ satisfies suitable conditions. We closely follow the 
proof in \cite{W} where the irreducibility of $M_{X, H}(r, c_1, c_2)$ is proved under
the assumption that $X$ is a smooth Hirzebruch surface. Hence we mainly explain the line of proof and the 
modifications needed to cover the singular case. Some details are omitted citing appropriate 
references to \cite{W}.

\begin{definition} \label{d3.1}
A coherent sheaf $E$ on $X$ is called prioritary (with respect to $p_1$) if it is
torsionfree and ${\rm Ext}^2(E,\, E(-F_y))\,=\,0$, where $F_y$ denotes the Cartier divisor
defined in Remark \ref{r2.3}. 
\end{definition} 

In the stack of coherent sheaves on $X$, the prioritary sheaves on $X$ are parametrized
by an open substack (by semicontinuity theorem). Let
$$Prior_X(r, c_1, c_2) \,\subset \,Coh_X(r, c_1, c_2)$$
denote the stack of priority sheaves on $X$ of rank $r$ and with fixed Chern classes
$c_1, c_2$. Let 
$$\text{H-SS}_X^{vect}(r, c_1, c_2) \,\subset\, Prior_X(r, c_1, c_2) $$
denote the substack of $H$-semistable prioritary vector bundles on $X$ of rank $r$
and Chern classes $c_1, c_2$.

For convenience (by abuse of notation), we denote $c_1(\omega_X\otimes {\mathcal O}_X(F_y))$ by 
$\omega_X+ F_y$ and $c_1(H)$ by $H$ again. Then the intersection (or cup product), of 
$c_1(\omega_X\otimes {\mathcal O}_X(F_y))$ with $c_1(H)$, evaluated on the fundamental cycle (or cap product with fundamental class) $[X]$ will be denoted by $H \cdot (\omega_X+ F_y)$. With these notations, we have the following lemma.

\begin{lemma}\label{l3.2}
If $H \cdot (\omega_X+ F_y))\,<\, 0$ for a general fiber
$F_y$ (Remark \ref{r2.3}), then any $H$-semistable sheaf $E$ is prioritary. 
\end{lemma}

\begin{proof}
Suppose that ${\rm Ext}^2(E, E(-F_y))\,\neq \,0$ for some non-singular point $y\,\in\,
Y$. Then by Serre duality, there exists a non-zero element
$\phi \,\in \, {\rm Hom} (E,\, E(\omega_X+F_y))$. Let $Im (\phi)$ denote the image of
the homomorphism $\phi$. By the $H$-semistability of $E$ and $E(\omega_X+F_y)$, we get
\begin{equation}\label{mu}
\mu_H(E) \le \mu_H(Im (\phi)) \,\le\, \mu_H(E(\omega_X+F_y))
\,=\, \mu_H(E)+H\cdot (\omega_X+F_y)\, .
\end{equation}
Since $H\cdot (\omega_X+ F_y)\,<\, 0$, this gives a contradiction. Thus $E$ is prioritary.

We note that the last equality in \eqref{mu} may not hold for a singular point $y$, hence
the proof fails if $F_y$ is not a general fiber. 
\end{proof}

\begin{lemma}\label{l3.3}
The stack $\text{\rm H-SS}_X^{vect}(r, c_1, c_2)$ is smooth.
\end{lemma}

\begin{proof}
It suffices to show that ${\rm Ext}^2(E,E)\,=\, 0$ for an $H$-semistable vector bundle
$E$. Let $y$ be a non-singular point of $Y$. One has a short exact sequence 
$$0\,\longrightarrow \,E(-F_y)\,\longrightarrow \,E\,\longrightarrow\, E\vert_{F_y}
\,\longrightarrow\, 0\, .$$
Applying Hom$(E, - )$ to this exact sequence yields
$$\longrightarrow \, {\rm Ext}^2(E, E(-F_y)) \,\longrightarrow \,
 {\rm Ext}^2(E,E)\,\longrightarrow \, {\rm Ext}^2(E, E\vert_{F_y})\,\longrightarrow \, .$$
We have $${\rm Ext}^2(E,\, E\vert_{F_y})\,\cong\, H^2(F_y, \,E^*\otimes E\vert_{F_y})
\,=\,0$$ as $F_y \,\cong\, {\mathbb P}^1$, and hence ${\rm
Ext}^2(E,\,E\vert_{F_y}) \,=\,0$. Since $E$ is prioritary, ${\rm Ext}^2(E,\,
 E(-F_y))\,=\,0$. It follows that ${\rm Ext}^2(E,\, E)\,=\,0$. 
\end{proof}

\begin{lemma}\label{l3.4}
Let $\sigma$ be a section of $p_1\,: \,X \,\longrightarrow\, Y$. Let $E$ be a coherent
sheaf on $X$ such that $(p_1)_* (E(-\sigma))\,=\, 0 \,=\, 
R^1(p_1)_* E$. Then there is an exact sequence 
$$0 \,\longrightarrow \, p_1^* (p_{1*}E)\,\longrightarrow \, E \,\longrightarrow \,
p_1^*(R^1p_{1*}(E(-\sigma))\otimes \Omega_{X/Y} (\sigma) 
\,\longrightarrow \, 0 \, .$$
\end{lemma}

\begin{proof}
This is essentially \cite[Lemma 8]{W}. The proof goes through in the singular case as 
\cite[Remark 3]{Bei} implies that the Beilinson resolution exists over any base $Y$.
\end{proof}

\begin{proposition}\label{p3.6}
The stack $Prior_X(r, c_1, c_2)$ of prioritary sheaves is irreducible. 
\end{proposition}

\begin{proof}
We briefly sketch a proof (see \cite[Proposition 2]{W} for details). Fix a section 
$\sigma \,\subset\, X$. Let $d\,=\, - c_1.F_y$. We may assume that $0 \,\le\, d\,<
\,r$ (by twisting $E$ with a power of ${\mathcal O}_X(\sigma)$). If $d\,>\,0$, we may
restrict ourselves to the dense open substack $Prior^0\,
\subset\, Prior_X(r, c_1, c_2)$ parametrizing all $E$ such 
that $E\vert_{F_y}\,\cong\, {\mathcal O}^{r-d}_{F_y} \oplus ({\mathcal O}_{F_y}(-1))^d$ 
(since the complement forms a closed substack of codimension at least one). If $d\,=\,0$, 
we may restrict ourselves to the dense open substack $Prior^0\,\subset\, Prior_X(r, c_1, 
c_2)$ parametrizing all $E$ such that $E\vert_{F_y} \,\cong\, {\mathcal O}^{r}_{F_y}$
for all but finitely many $y \,\in\, Y$ and $E\vert_{F_y} \,\cong\,
{\mathcal O}_{F_y}(1) \oplus {\mathcal O}^{r-2}_{F_y} \oplus ({\mathcal O}_{F_y}(-1))$
at these finitely many points.

In either case, one sees that $K\,=\, p_{1*} E$ is a vector bundle on $Y$ of rank $r-d$ 
and degree $k\,=\, \chi(E) +(r-d)(g-1)$. Also $L\,=\, R^1p_{1*} (E( - \sigma))$ is a sheaf of 
rank $d$, degree ${\ell}\,=\, - \chi(E) +(c_1.\sigma) - (r-d)(g-1)$. Moreover, $L$ is 
locally free for for $d\,>\,0$ and a skyscraper sheaf for $d\,=\,0$.

By Lemma \ref{l3.4}, there exists a short exact sequence
$$0\,\longrightarrow \,p_1^*K\,\longrightarrow \, E \,\longrightarrow \,
p_1^*L \otimes \Omega_{X/Y} (\sigma) \,\longrightarrow \, 0\, .$$
By \cite[p. 200]{LeP}, if $E$ has a filtration 
$${\mathcal F} \,: \ 0 \,=\,F_0 \,\subset\, F_1\,
\subset\, \cdots \,\subset\, F_t\,=\, E\, , ~\ gr_i(E)\,:=\, (E_i/E_{i-1})\, , $$
one defines groups ${\rm Ext}^i_{+} (E,\,E)$ and ${\rm Ext}^i_{-} (E,\,E)$ such that there
is an exact sequence 
$$
\cdots \,\longrightarrow \, {\rm Ext}^i_{-} (E,\,E)\,\longrightarrow \,{\rm Ext}^i (E,\,E)
\,\longrightarrow \, {\rm Ext}^i_{+} (E,\,E)\,\longrightarrow \, \cdots \, .$$
Then there is a spectral sequence for which
$$E_1^{p,q}\,=\, \prod_i \ {\rm Ext}^{p+q} (gr_i(E), \,gr_{i-p}(E)) ~ \ {\rm if} \ 
p\,<\,0;~ E_1^{p,q}\,=\,0 ~ \ {\rm for} \ p\,\ge\, 0\, ,$$ 
and which converges to ${\rm Ext}^{\bullet}_{+} (E,\,E)$.

In our case, for ${\mathcal F} \,:\, 0 \,\subset\,\pi^*K\,\subset\, E$, we have 
$${\rm Ext}^i_{+} (E,\,E)\,=\, H^i(p_1^*(K^*\otimes L) \otimes \Omega_{X/Y}(\sigma)) 
\,=\, H^i(K^*\otimes L\otimes p_{1*}{\mathcal O}_{{p_1}} (-1)) \,=\,0 \, ~\forall
~ i \, .$$
Hence ${\rm Ext}^i_{-} (E,\,E)\,= \,{\rm Ext}^i (E,\,E)$ for all $i$, so that infinitesimal
deformations of $E$ are same as those of $0\,\subset\, \pi^*K\,\subset\, E$.

By \cite[Remark on p. 201]{LeP}, there is a spectral sequence with 
$$E_1^{p,q}\,=\, \prod_i \, {\rm Ext}^{p+q} (gr_i(E),\, gr_{i-p}(E)) \ {\rm if} \ p
\,\ge\, 0;~ E_1^{p,q}\,=\,0 \ {\rm for} \ p \,<\, 0\, ,$$ 
and which converges to ${\rm Ext}^{\bullet}_{-} (E,\,E)$.

Since $${\rm Ext}^2 (p_1^* L \otimes \Omega_{X/Y} (\sigma), \, p_1^*K)\,=\, 0\, ,$$ 
straightforward computations show that
$${\rm Ext}^1 (E,\,E)\,=\, {\rm Ext}^1_{-} (E,\,E)$$ surjects onto ${\rm Ext}^1 (K,\,K)
\oplus {\rm Ext}^1 (L,\,L)$. Hence a general infinitesimal deformation of $E$ induces a
general infinitesimal deformation of $K$ and $L$, and the morphism 
$$\phi\,:\, Prior^0\,\longrightarrow \, Coh_Y(r-d,k) \times Coh_Y(d,{\ell})\, ,
~ \ [E] \,\longmapsto\, ([K],\, [L]) $$
is dominant. Since the stack of vector bundles of fixed rank and degree on a nodal curve
is irreducible, \cite{Re}, and every coherent sheaf is in the limit of vector bundles, the
stack of coherent sheaves of fixed rank and degree on a nodal curve is irreducible. Hence
the image of $\phi$ is irreducible. The fibers of $\phi$ are stack quotients of an affine
subscheme of the affine space ${\rm Ext}^1 (p_1^* L \otimes \Omega_{X/Y} (\sigma),\,
p_1^*K)$, hence they are irreducible. It follows that $Prior_X(r, c_1, c_2)$ is irreducible. 
\end{proof}

Lemma \ref{l3.3} and Proposition \ref{p3.6} together give the following:

\begin{corollary}\label{c3.7}
The substack $\text{\rm H-SS}^{vect}_{ X}(r,c_1,c_2)$ is a smooth irreducible open
substack of the stack $Prior_X(r, c_1, c_2)$.
\end{corollary}

\begin{theorem} \label{thm1}
The moduli space $M_{X,H}(r,c_1,c_2)$ of $H$-semistable vector bundles on $X$ is a
normal irreducible variety. Its open subscheme corresponding to stable vector bundles is
a smooth variety. 
\end{theorem}

\begin{proof}
This can be proved on similar lines as the corresponding part of the proof
of \cite[p. 208,
Theorem 1]{W}. The smooth irreducible stack $\text{H-SS}^{vect}_X$
(see Corollary \ref{c3.7}) is a quotient stack $[Q^{ss}/{\rm GL}(N)]$, where $Q^{ss}$ is an open
subscheme of a quot scheme. Hence $Q^{ss}$ is a smooth irreducible scheme. The moduli
space $M_{X,H}(r,c_1,c_2)$ is the GIT quotient $Q^{ss}/\!\!/ {\rm PGL}(N)$ for
Simpson's polarization on $Q^{ss}$ (see \cite{Si}). Hence $M_{X,H}(r,c_1,c_2)$ is a normal
irreducible variety and its open subscheme corresponding to stable points, i.e., the open
subscheme corresponding to stable vector bundles, is a smooth variety.
\end{proof}

\begin{lemma}
There exists an ample line bundle $H$ on $X$ such that $$H \cdot \omega_X + F_y\, <\,0\, .$$
\end{lemma}

\begin{proof}
This follows essentially imitating the proof of the first part of \cite[Lemma 10]{W}. 
Let the notations be as in Lemma \ref{l3.2}. By Lemma \ref{l2.2}, we may take
$H\,=\, b F_y + \sigma$, $b\,>>\, 0$, where $\sigma$ is
the class of ${\mathcal O}_{p_1}(1)$. One has $$F_y\cdot F_y\,=\,0\, , ~\ F_y\cdot \sigma
\,=\,1\ \text{ and }~\ \sigma \cdot \sigma \,=\, {\mathcal O}_{p_0}(1)\cdot
{\mathcal O}_{p_0}(1)\,=\, - e$$ by \cite[Chapter V, Proposition 2.9]{Ha}.
Using Proposition \ref{p2.4} it follows that the class
of $\omega_X+F_y$ is $(2g-1-e)F_y - 2 \sigma$. Hence 
$$
\begin{array}{ccl}
H \cdot (\omega_X+F_y) &  =  &  (b F_y + \sigma)((2g-1-e) F_y - 2 \sigma \\
{} & =  &  -2b + 2g -1 - e + 2 e\\
{} & =  &  -2b +2g-1 - e. 
\end{array}
$$
It follows that $H \cdot (\omega_X+F_y)\, <0$ for $b >>0$, i.e., $(c_1(H) \cup c_1(\omega_X+
{\mathcal O}_X(F_y)))\cap [X]\, <\,0$ for $b >>0$.
\end{proof}

\section{Rationality of $M_X(r,c_1,c_2)$}

For a coherent sheaf $W$ on projective variety 
${\mathcal M}$ equipped with a very ample line bundle ${\mathcal H}$, define
$W(m)\,:= \,W\otimes {\mathcal H}^{\otimes m}$. Let
$P_{\mathcal M}(W,m)$ denote the Hilbert polynomial of $W$ with respect to $\mathcal
H$ \cite{HL}. Then
$$P_{\mathcal M}(W,m)\,=\, \sum _{i=0}^{\dim supp.(W)} a_i(W) \frac{m^i}{i!}\, , \ a_i(W) \ 
{\rm integers}\, .$$
The rank of $W$ is $r(W)\,=\, \frac{a_d(W)}{a_d({\mathcal O}_{\mathcal M})},~ d\,=\,
\dim \mathcal M$,
and the degree of $W$ is $d(W)\,= \,a_{d-1}(W)- r(W)a_{d-1}({\mathcal O}_{\mathcal M})$.

By the projection formula, 
$$(p_*E)(m)\,:=\, p_*E \otimes H^{\otimes m}\,\cong\, p_*(E
\otimes p^*H^{\otimes m})\,=\,p_*(E \otimes H_Z^{\otimes m})\,=\, p_*(E(m))\, .$$
Hence $H^i(X, (p_*E)(m))\, =\, H^i(X, p_*(E(m)))\, , \forall i$. Since $p$ is a finite map, 
$$H^i(X, p_*(E(m)))\, =\, H^i(Z, E(m))\, .$$ It follows that 
\begin{equation}\label{pxz}
P_X(p_*E , m)\, =\, P_Z(E, m)\, . 
\end{equation}

One has a short exact sequence 
\begin{equation}\label{ox} 
 0\,\longrightarrow\, {\mathcal O}_X\, \longrightarrow\, p_* {\mathcal O}_Z
\,\longrightarrow\, \bigoplus_{j=1}^{g}{\mathcal O}_{P_j}\,\longrightarrow\, 0\, .
\end{equation}
Tensoring \eqref{ox} with $H^{\otimes m}$ and using \eqref{pxz}, we see that 
the Hilbert polynomials satisfy the equation 
$$P_X({\mathcal O}_X, m) \,=\, P_Z({\mathcal O}_Z,m)+ \sum_j \chi({\mathcal O}_{P_j}(m))\, .$$
Since $d(H\vert_{P_j})\,=\,\ell$, we have $\chi({\mathcal O}_{P_j}(m))\,= \,\ell m +1$.
Hence comparing the coefficients of powers of $m$ it follows that
\begin{equation}\label{aox}
a_2({\mathcal O}_X)\,=\,a_2({\mathcal O}_Z), \ a_1({\mathcal O}_X)\,=\,
 a_1({\mathcal O}_Z) +\ell g\, ,~\ \ a_0({\mathcal O}_X)\,=\,a_0({\mathcal O}_Z)+ g \, .
\end{equation}

\subsection{Generalized parabolic bundles}\label{4.1}

Here we give a definition of a generalized parabolic bundle (GPB for short) suitable
in our special case. For a more general definition of a GPB and generalities on them
the reader may refer to \cite[Section 2]{B0}.
 
Let $F$ be a vector bundle on $X$ and $E\,:=\, p^*F$ its pullback to $Z$.
Since $E\vert_{P_{x_j}}
\,\cong\, p^*(F\vert_{P_j}) \,\cong\, E\vert_{P_{z_j}}$, we get a canonical isomorphism
$$\sigma_j\,:\, E\vert_{P_{x_j}} \,\stackrel{\sim}{\longrightarrow}\, E\vert_{P_{z_j}}$$
lying over the isomorphism $\tau_j$ in \eqref{tau}. Let $\sigma\,:=\,
 (\sigma_1\, , \cdots\, ,\sigma_g)$. Thus a vector bundle $F$ on $X$ determines a pair
$(E\, , \sigma)$ as above. We call such a pair a generalized parabolic bundle.

Conversely, given a GPB $(E\, , \sigma)$ on $Z$, we get a vector bundle $F$ on $X$
in the following way. Let $F_j(E)\,\subset\, E\vert_{P_{x_j}} \oplus E\vert_{P_{z_j}}$
denote the graph of $\sigma_j$. The surjective morphism $E\,\longrightarrow\,
E\vert_{P_{x_j}} \oplus E\vert_{P_{z_j}}$ produces a surjection of 
${\mathcal O}_X$--modules 
$$p_*E\, \longrightarrow\, \bigoplus_j
p_*( (E\vert_{P_{x_j}} \oplus E\vert_{P_{z_j}})/F_j(E))\, .$$
 Let $F$ be the kernel of the
latter surjection, so $F$ fits in the exact sequence 
$$0\,\longrightarrow\, F\,\longrightarrow\, p_*E\,\longrightarrow\, \bigoplus_j
p_*( (E\vert_{P_{x_j}} \bigoplus E\vert_{P_{z_j}})/F_j(E) )\,\longrightarrow\, 0\,.$$
One sees that $p^*F \,=\,E$, and hence $E$ and $F$ have the same rank and same Chern
classes. The above construction gives a bijective correspondence between GPBs of rank
$r$ and Chern classes $c_1, c_2$ on $Z$ and vector bundles of rank $r$ and Chern
classes $c_1, c_2$ on $X$.
 
\begin{lemma} \label{lemma2}
Let $(E\, ,h)$ be a GPB on $Z$ determining a vector bundle $F$ on $X$. Let $E_1\,\subset
\,E$ be a torsion free subsheaf such that $E/E_1$ is torsionfree. Then $E_1$ determines
a subsheaf $F_1\,\subset\, F$ such that $\mu(F_1)\,\le \,\mu(E_1)$.
\end{lemma} 

\begin{proof}
Since the quotient $E/E_1$ is torsion free, it follows that $$E_1\vert_{P_{x_j}}
\oplus E_1\vert_{P_{z_j}}\,\subset\, E\vert_{P_{x_j}} \oplus E\vert_{P_{z_j}}\, .$$ Let 
$$F_j(E_1)\,:=\, F_j(E) \cap (E_1\vert_{P_{x_j}} \oplus E_1\vert_{P_{z_j}})\, ,~\
\ Q_j(E_1)\,:=\, (E_1\vert_{P_{x_j}} \oplus E_1\vert_{P_{z_j}})/ F_j(E_1)\, . $$
Define the sheaf $F_1$ on $X$ by the following short exact sequence
\begin{equation} \label{gpb1}
0\,\longrightarrow\, F_1\,\longrightarrow\, p_*E_1\,\longrightarrow\, \oplus_j p_*Q_j(E_1)
\,\longrightarrow\, 0\, .
\end{equation}
Since $h$ is an isomorphism, the projection $pr_j\,:\, F_j(E)\,\longrightarrow\,
E\vert_{P_{x_j}}$ is an isomorphism. Therefore, $$pr_j\vert_{F_j(E_1)}\,\longrightarrow
\, E_1\vert_{P_{x_j}}$$ is an injection. Hence 
$r(F_j(E_1)) \,\le\, r(E_1\vert_{P_{x_j}})$. Similarly, we have $r(F_j(E_1))
\,\le \, r(E_1\vert_{P_{z_j}})$. 
Hence tensoring \eqref{gpb1} by $H^{\otimes m}$, we have
\begin{equation} \label{gpb2}
0 \,\longrightarrow\, F_1(m)\,\longrightarrow\, (p_*E_1)(m)\,\longrightarrow\,
\oplus_j p_*Q_j(E_1)(m)\,\longrightarrow\, 0\, .
\end{equation} 
Therefore, $P_X(F_1,m)\,=\, P_X(p_*E_1,m) - \sum_j P_{{P_j}} (p_*(Q_j(E_1),m))$. By \eqref{pxz}, 
$$P_X(p_*E_1, m)\,=\, P_Z(E_1,m))~\ \text{ and }~\ P_{{P_j}} (p_*Q_j(E_1),m)
\,=\, \chi_{(P_j)}(p_*Q_j(E_1),m)\, .$$ Hence comparing coefficients of $m$ in
\eqref{gpb2}, we conclude that
$$a_1(E_1) \,=\, a_1(F_1) + b_1\, ,$$
where 
$$b_1\,=\, \sum_j r(p_*(Q_j(E_1)(m))) \ell \,=\, \sum_j r(p_*(Q_j(E_1))) \ell 
$$
$$
=\, \sum_j \ell (r(E_1\vert_{P_{x_j}}) + r(E_1\vert_{P_{z_j}}) - r(F_j(E_1)))\, .$$
Since $r(F_j(E_1))\,\le\, r(E_1\vert_{P_{x_j}})$, we have 
$$b_1\,\ge\, \sum_j \ell (r(E_1\vert_{P_{z_j}}))\, .$$ Similarly, 
$b_1\,\ge \,\sum_j \ell (r(E_1\vert_{P_{x_j}})$, and hence 
$$b_1\,\ge\, \ell \sum_j \ {\rm max} \ \{r(E_1\vert_{P_{x_j}})\,,\,
(r(E_1\vert_{P_{z_j}})\}\, .$$
Then $a_1(E_1)\,=\, a_1(F_1) + b_1\,\ge\, a_1(F_1)+ \ell \sum_j \ {\rm max} \
\{r(E_1\vert_{P_{x_j}})\, ,\, (r(E_1\vert_{P_{z_j}}) \}\, .$
By definition,
$$
\begin{array}{lcl}
d(E_1)& = & a_1(E_1) -r(E)a_1({\mathcal O}_X) \\
{}  & \ge & a_1(F_1)+ \ell \sum_j \ {\rm max} \ \{ r(E_1\vert_{P_{x_j}})\,,\,
(r(E_1\vert_{P_{z_j}}) \} - r(F_1)a_1({\mathcal O}_X)   \\
{} & = & a_1(F_1)+ \ell \sum_j \ {\rm max} \ \{ r(E_1\vert_{P_{x_j}})\,,
\, (r(E_1\vert_{P_{z_j}}) \} - r(F_1)a_1({\mathcal O}_Z) - \ell g r(F_1)   \\
{} & = & d(F_1) + \ell \sum_j ( \ {\rm max} \ \{ r(E_1\vert_{P_{x_j}})\, ,\,
(r(E_1\vert_{P_{z_j}}) \} -r(F_1))    \\ 
{} & \ge & d(F_1)\, .
\end{array}
$$
Since $r(F_1) \,=\, r(E_1)$, the result follows.
\end{proof}

\begin{proposition} \label{prop3}
If $E$ is an $H_Z$-semistable (respectively, $H_Z$-stable) vector bundle on $Z$ with $(E\, ,h)$
giving a vector bundle $F$ on $X$, then $F$ is $H$-semistable (respectively, $H$-stable).
\end{proposition}

\begin{proof}
Let $F'_1\,\subset\, F$ be a torsion free subsheaf. Then we have $(p^*F'_1/{\rm Torsion})\,
\subset\, E$. Let $E_1$ be the saturation of $(p^*F'_1/{\rm Torsion})$ in $E$. 
Then $E/E_1$ is torsionfree and $E_1$ gives a torsionfree subsheaf $F_1\,\subset\, F$ such
that $F'_1\,\subset\, F_1$. Since the quotient $F_1/F'_1$ is a torsion sheaf,
we have $\mu(F'_1)\,\le \,\mu(F_1)$. By Lemma \ref{lemma2}, $\mu(F_1)\,\le\,\mu(E_1)$ so that
$\mu(F'_1) \le \mu(E_1)$. If $E$ is semistable (respectively, stable), then $\mu(E_1)\,\leq\,
\mu(E)$ (respectively, $\mu(E_1)\, <\, \mu(E)$) so that 
$$
\mu(F'_1) \,\le\, \mu(E_1) \,\leq\, \mu(E)\,=\, \mu(F)\ {\rm (respectively,}\
\mu(F'_1) \,\le\, \mu(E_1) \,<\, \mu(E)\,=\, \mu(F){\rm )}
$$
proving the proposition. 
\end{proof}

\begin{theorem}\label{theorem2}
Let $M^s_{X,H}(r,c_1,c_2)$ be the moduli space of $H$-stable (slope stable) vector bundles of rank $r$ and Chern classes $c_1, c_2$ on $X$. Let $\Delta(r,c_1,c_2)= \frac{1}{r}(c_2- \frac{r-1}{2r} c_1^2)$. Let $F_Z$ denote the general fiber of $Z$. If $\Delta(r,c_1,c_2) >>0$, then $M^s_{X,H}(r,c_1,c_2)$ is a non-empty, smooth, irreducible, rational, quasiprojective variety in the following cases ($c_1, c_2$ below denote the classes on $Z$ which are pull-backs of the classes $c_1, c_2$ on $X$):\\
(1) $c_1.F_Z = 1 \ {\rm or} \ r-1 \ ( {\rm mod} \ r)$.\\
(2) $c_1.F_Z = r-2 \  ( {\rm mod} \ r) \ {\rm and} \ c_2-c_1^2/2 -c_1.\omega_Z/2 -(r-1)=0 
  \ ( {\rm mod} \ 2)$.\\ 
(3) $c_1.F_Z = 2 \ ( {\rm mod} \ r) \ {\rm and} \ c_2 + c_1.C_0 -c_1^2/2 + c_1.\omega_Z/2 +1 =0 
  \ ( {\rm mod} \ 2)$.\\
\end{theorem}

\begin{proof}
In \cite{CM}, Costa and Mir\'o-Roig construct explicitly generic $H$- stable vector 
bundles $E$ of rank $r$ and Chern classes $c_1, c_2$ on $Z$ in the above cases. It is 
easy to see that in each of these cases, the restrictions of $E$ to all fibers are 
isomorphic. Using the bijective correspondence between GPBs on $Z$ and vector bundles on 
$X$, together with Proposition \ref{prop3}, we can construct (generic) $H$-stable vector 
bundles on $X$ of rank $r$ and Chern classes $c_1, c_2$. Thus $M^s_{X,H}(r,c_1,c_2)$ is 
non-empty in all the cases. By Theorem \ref{thm1}, this moduli space is irreducible and 
smooth.

We now turn to the question of rationality of $M^s_{X,H}(r,c_1,c_2)$. We shall in fact 
show that whenever $M^s_{Z,H_Z}(r,c_1,c_2)$ is rational, the variety 
$M^s_{X,H}(r,c_1,c_2)$ is also rational. Since the rationality of 
$M^s_{Z,H_Z}(r,c_1,c_2)$ is known in the cases listed in the statement of the theorem 
\cite{CM}, this will prove the theorem.

The moduli space $M_{Z,H_Z}(r,c_1,c_2)$ is a geometric invariant theoretic (GIT) 
quotient of a suitable quot scheme $Q^{ss}$ by ${\rm PGL}(N)$. Let $Q^s$ and 
$M^s_{Z,H_Z}(r,c_1,c_2)$ denote the subschemes corresponding to stable vector bundles, 
then $M^s_{Z,H_Z}(r,c_1,c_2) \,=\, Q^s /\!\!/ {\rm PGL}(N)$. Let
$${\mathcal U}_Q \,\longrightarrow \, Q^s \times Z$$
be the universal quotient vector bundle.

For simplicity of exposition, let us take $g\,=\,1$ and write $x_1\,=\,x\, ,\ z_1\,
=\,z$. We have vector bundles 
${\mathcal U}_Q\vert_{Q^s\times P_x}$ and $(id \times \tau)^*({\mathcal U}_Q\vert_{Q^s
\times P_z})$ on $Q^s \times P_x$. Consider the sheaf 
$${\mathcal H}_Q\,=\, Hom ({\mathcal U}_Q\vert_{Q^s\times P_x}, \ (id \times \tau)^*({\mathcal U}_Q\vert_{Q^s\times P_z}))\, .$$
Since scalars (the isotropy) acts trivially on the sheaf ${\mathcal H}_Q$, it descends to
a sheaf ${\mathcal H}_{M,x}$ on $M^s_{Z,H_Z}(r,c_1,c_2)\times P_x$. Let 
$${\widetilde H}_x \,=\, Rp_{M_Z *} {\mathcal H}_{M,x}$$
denote the direct image of ${\mathcal H}_{M,x}$ on $M^s_{Z,H_Z}(r,c_1,c_2)$.  There is
a Zariski open subset $M' \,\subset\, M^s_{Z,H_Z}(r,c_1,c_2)$ such that
${\widetilde H}_x\vert_{M'}$ is locally free (by semi-continuity theorem). Hence there is
a Zariski open subset $M'' \,\subset\, M'$ such that
 $$
 {{\widetilde H}_x}\vert_{M''}\,\cong\, M'' \times {\mathbb C}^n\, .$$
 Therefore, if $M^s_{Z,H_Z}(r,c_1,c_2)$ is rational, then $M''$ and hence the total
space of ${{\widetilde H}_x}\vert_{M''}$ is  rational.
 
Note that the fiber of ${{\widetilde H}_x}\vert_{M'}$ over $[E] \,\in\, M'$ corresponds to the 
vector space $${\rm Hom} (E\vert_{P_x}\,,\, \tau^*(E\vert_{P_z}))\, .$$ Since ${\rm Hom} 
(E\vert_{P_x}, \tau^*(E\vert_{P_z}))\,\supset\, {\rm Iso} (E\vert_{P_x}, 
\tau^*(E\vert_{P_z})),$ there is a Zariski open subset ${\widetilde H}'$ of the total space 
of ${{\widetilde H}_x}\vert_{M''}$ which corresponds to generalized parabolic bundles on 
$Z$. By Proposition \ref{prop3} and Section \ref{4.1}, this ${\widetilde H}'$ is 
isomorphic to an open subset of $M_{X,H}(r,c_1.c_2)$.  It now follows that 
$M_{X,H}(r,c_1,c_2)$ is rational.

In the case of $g\,>\,1$, we take the sheaf 
 \begin{equation}\label{HM}
 {\mathcal H}_M \ \longrightarrow \  M^s_{Z,H_Z}(r,c_1,c_2)\times \coprod_j  P_{x_j}
\,= \,\coprod _j (M^s_{Z,H_Z}(r,c_1,c_2) \times P_{x_j})
\end{equation} 
 to be the sheaf whose restriction to $M^s_{Z,H_Z}(r,c_1,c_2) \times P_{x_j} $ is
${\mathcal H}_{M,x_j}$. There is a Zariski open subset $S\,\subset\, M^s_{Z,H_Z}(r,c_1,c_2)$
such that the restriction of the direct image of ${\mathcal H}_M$ to $S$ is a vector
bundle. The rest of the argument works as in the one node case. 
\end{proof}

\section{Vector bundles over a real ruled surface}

In this section, we study moduli of vector bundles over a real rational ruled 
surface. Our goal is to prove that the moduli space $M(r, c_1,c_2)$ for vector bundles 
over a real rational ruled surface is rational as a real variety.

\subsection{The real rational ruled surface} \hfill

Let $\sigma$ be an anti-holomorphic involution on ${\mathbb P}^1_{\mathbb C}$. The pair 
$({\mathbb P}^1_{\mathbb C}\, , \sigma)$ defines a smooth projective curve $C$ defined
over $\mathbb R$. 
Let $(x_1\, ,z_1)\, , \cdots\, , (x_g\, , z_g)$ be distinct pairs of points of ${\mathbb 
P}^1_{\mathbb C}$. Let $Y$ be a complex nodal curve of genus $g(Y)\,=\,g$ obtained by 
identifying points $x_j$ with $z_j$ for each $j$ such that the involution $\sigma$ 
induces an anti-holomorphic involution $\sigma_Y$ on $Y$. We clarify that $\sigma_Y$
need not fix pointwise $\{x_j\}_{j=1}^g$ and $\{z_j\}_{j=1}^g$. Note that
$\sigma_Y(x_i)\, \in\, \{x_j\, ,z_j\}$ if and only if $\sigma_Y(z_i)\, \in\,
\{x_j\, ,z_j\}$.
Let $y_1\, , \cdots \, , y_g$ denote the nodes of $Y$ with $x_j\, , z_j$ identified to
$y_j$. The pair $(Y\, , \sigma_Y)$ defines a projective curve $Y_{\mathbb R}$, of
arithmetic genus $g$, defined over ${\mathbb R}$.  We have 
$$Y_{\mathbb R} \times_{\mathbb R} {\mathbb C}\,=\,Y\, .$$
The  normalization of $Y_{\mathbb R}$ is the curve $C$,  let 
$$\pi'\,:\, C\, \longrightarrow\, Y_{\mathbb R}$$
be the normalization map. If $C$ does not have any real points, then all real points
of $C$ lie in $\{y_j\}_{j=1}^g$.
 
Let ${\mathcal E}_{\mathbb R}$ be a real algebraic vector bundle
of rank two over $Y_{\mathbb R}$,
and let ${\mathcal E}\,= \,{\mathcal E}_{\mathbb R}\otimes_{\mathbb R} {\mathbb C}$ be its
base change to $\mathbb C$, which is a complex vector bundle over $Y$. The vector bundle
${\mathcal E}_{\mathbb R}$ is defined by a pair $({\mathcal E}\, ,
\sigma_{\mathcal E})$. Then $\sigma_{\mathcal E}$ induces an anti-holomorphic involution
$\sigma_X$ on $X\, := \mathbb{P}(\mathcal{E}) $. The pair $(X\, , \sigma_X)$ defines the real ruled surface 
$$X_{\mathbb R}\,:=\, {\mathbb P}({\mathcal E}_{\mathbb R}) \longrightarrow  Y_{\mathbb R}\, ,$$ 
and one has $X_{\mathbb R} \times_{\mathbb 
R} {\mathbb C} \,=\, X$.  Note that the anti-holomorphic involution $\sigma_X$ lifts the 
anti-holomorphic involution $\sigma_Y$ of $Y$. Since $\sigma_Y$ permutes the nodes of
$Y$, the  involution $\sigma_X$ permutes the fibers $\{P_j\}_{j=1}^g$, so
$\coprod_j P_j$  is  $\sigma_X$--invariant, thus $\coprod_j P_{j}$ is a real variety.
 
We have $Z \,=\, {\mathbb P}(\pi^* {\mathcal E})$; let  $\sigma_Z\, :=\, p^*\sigma_X$
be the anti-holomorphic involution induced by $\pi^*\sigma_{\mathcal E}$. The pair
$(Z\, , \sigma_Z)$ defines a ruled surface over $C$, defined over ${\mathbb R}$. 
Let $P\,:= \,\coprod_j P_{x_j}$.  Since we can canonically identify $P_j$ with $P_{x_j}$ for each $j$, we see 
 that $\sigma_Z$ leaves $P$ invariant. Let
$$
\sigma_P\, :=\, \sigma_Z\vert_P\,:\, P\,\longrightarrow\, P
$$
be the restriction.

The anti-holomorphic involution $\sigma_Z$ on $Z$ produces an anti-holomorphic involution
$\sigma_M$ on $M^s_{Z,H_Z}(r,c_1,c_2)$ defined by 
$$\sigma_M (E)\,=\, \sigma_Z^*({\overline E})\, .$$
 There is a sheaf ${\mathcal H}_M$ on $M^s_{Z,H_Z}(r,c_1,c_2)\times
 P\,=\,\coprod _j (M^s_{Z,H_Z}(r,c_1,c_2) \times P_{x_j})$ (defined by \eqref{HM} in the 
proof of Theorem \ref{theorem2}).  For a general vector bundle $E\,\in\,
 M^s_{Z,H_Z}(r,c_1,c_2)$, one has
$E\vert_{P_{x_j}}\,\cong\, \tau_j^* (E\vert_{P_{z_j}})$ for all $j$,
where $\tau_j$ is defined in \eqref{tau}. We  choose a
$\sigma_M$-invariant open subset $M'\,\subset\, M^s_{Z,H_Z}(r,c_1,c_2)$ such that $E\,\in
\,M'$ satisfies the above condition and ${\mathcal H}_M\vert_{M'\times P}$ is locally free. 

\begin{lemma}\label{l4.1}
The vector bundle ${\mathcal H}_M\vert_{M'\times P}$ is a real vector bundle. 
\end{lemma}  

\begin{proof}
We shall construct an anti-holomorphic involution $\sigma_H$ on ${\mathcal 
H}_M\vert_{M'\times P}$ lifting $\sigma_M\times \sigma_P$. For $(E, v_j) \,\in\, M' \times 
P_{x_j}$, we have 
$$(\sigma_M\times \sigma_P) (E, v_j)\,=\, (\sigma_Z^*{\overline E}, 
\sigma_P(v_j))\,\in\, M' \times \sigma_P(P_{x_j})$$ 
and $\sigma_Z^*{\overline 
E}\vert_{\sigma_P(P_{x_j})} \,=\, {\overline E}\vert_{P_{x_j}}$. Write 
$$E_1\,=\, E\vert_{P_{x_j}}\, ,\ 
E_2\,=\, \tau^*(E\vert_{P_{z_j}})\, .$$ 
One has ${\mathcal H}_M\vert_{E\times P_{x_j}} \,= \,Hom (E_1, E_2)$ and
$${\mathcal H}_M\vert_{\sigma_M E \times \sigma_P(P_{x_j})}\,=\,
Hom ({\overline E}\vert_{P_{x_j}}, \tau^*({\overline E}\vert_{P_{z_j}}))\,=\,
Hom (\overline{E_1}, \overline{E_2})\, .
$$
Any linear homomorphism $f\,:\, E_1\,\longrightarrow\, E_2$ induces a linear
homomorphism $$\overline{f}
\,: \,\overline{E_1}\,\longrightarrow\, \overline{E_2}$$ such that $\overline{(\overline{f})}
\, =\, f$. Hence there is a natural anti-holomorphic involution $\sigma_H$ on
${\mathcal H}_M\vert_{M'\times P}$ which lifts $\sigma_M\times \sigma_P$ and
$$\sigma_H \,:\, Hom (E_1,\, E_2)\, \longrightarrow\, Hom (\overline{E_1}, \,
\overline{E_2})$$ is defined by $f\,\longmapsto
\, \overline{f}$. Since  $\overline{(\overline{f})}\,=\, f$, it follows that
${\sigma_H}^2 \,=\,{\rm Id}$. Hence ${\mathcal H}_M\vert_{M'\times P}$ is a real vector bundle.
\end{proof}
 
\begin{theorem}\label{theorem3}
The variety $M_{X,H}(r,c_1,c_2)$ is rational as a real variety if $M_{Z,H_Z}(r,c_1,c_2)$ is rational as a real variety.  
\end{theorem} 
 
\begin{proof}
Since ${\mathcal H}_M\vert_{M'\times P}$ is a real vector bundle (see Lemma \ref{4.1}), so is its direct image on $M'$. Let 
$$V \ = \ {p_{M'}}_* ({\mathcal H}_M\vert_{M'\times P})\, .$$
 The involution $\sigma_H$ induces an anti-holomorphic involution $\sigma_V$ on $V$. By replacing $M'$ by a $\sigma_M$-invariant open subset if necessary, we see that there is a Zariski open subset $U$ of the total space of $V$ such that for $E\in M'$ the fiber $U_E = \bigoplus_j {\rm Iso} (E_1, E_2)$ and
$U\, \longrightarrow\, M'$ is a locally trivial fiber bundle with fibers isomorphic to an (fixed) affine space. The involution $\sigma_V$ keeps $U$ invariant. Thus $U$ is a real variety. 
 
Since $M'$ is rational as a real variety by our assumption, the above subset $U$ is also 
rational as a real variety. The real variety $M_{X,H}(r,c_1,c_2)$ has an open 
subset isomorphic to $U$. It follows that $M_{X,H}(r,c_1,c_2)$ is rational as a 
real variety.
\end{proof}

\begin{theorem}\label{theorem4}
Let $\sigma$ be an anti-holomorphic involution on ${\mathbb P}^1_{\mathbb C}$ defined by 
 $$ \sigma(x:y)\,=\,({\overline y}, - {\overline x})\, .$$
The pair $({\mathbb P}^1_{\mathbb C}, \sigma)$ defines a non-degenerate anisotropic conic $C$ 
over ${\mathbb R}$. 

Let $c_1= C_0 + d F_Z, F_Z$ being the general fiber of $Z$.  Let $c_2, \alpha, \lambda$  and $m$ be integers satisfying 
$$c_2\,=\, \lambda (r-1) + \alpha\, ,\, \ 0 \,<\, \alpha\,\le\, r-1\, ,\ m
\,=\, d- c_2 - 1 - \lambda\, .$$
Assume that 
$$\Delta(r,c_1,c_2)\,:=\, \frac{1}{r} (c-2 - \frac{r-1}{2r} c_1^2)\,>>\, 0\, .$$ 
Then $M_{X,H}(r,c_1,c_2)$ is rational as a real variety if and only if one of the following conditions holds:
\begin{enumerate}
\item Both the integers $m$ and $r-1- \alpha$ are even.\\ 
\item The integer $m$ is odd and $\alpha$ is even.
\end{enumerate}
\end{theorem}  
\begin{proof}
The theorem follows from Theorem \ref{theorem3} and the main theorem of \cite{BS}. 
\end{proof}

\end{document}